\documentclass[12pt,reqno]{amsart}

\usepackage{amssymb}

\newcommand{\alp}{\alpha}
\newcommand{\del}{\delta}
\newcommand{\gam}{\gamma}
\newcommand{\kap}{\varkappa}
\renewcommand{\phi}{\varphi}

\newcommand{\R}{{\mathbb R}}
\newcommand{\Z}{{\mathbb Z}}

\newcommand{\e}{{\rm e}}

\newcommand{\cC}{{\mathcal C}}
\newcommand{\cS}{{\mathcal S}}

\newcommand{\hA}{{\widehat A}}

\newcommand{\Zp}{\Z/p\Z}

\newcommand{\lpr}{\left(}
\newcommand{\rpr}{\right)}
\newcommand{\lfl}{\left\lfloor}
\newcommand{\rfl}{\right\rfloor}

\newcommand{\stm}{\setminus}
\newcommand{\seq}{\subseteq}
\newcommand{\est}{\varnothing}

\newcommand{\longc}{,\ldots,}
\newcommand{\longe}{=\dotsb=}

\newtheorem{lemma}{Lemma}
\newtheorem{claim}{Claim}
\newtheorem{theorem}{Theorem}

\newcommand{\refl}[1]{~\ref{l:#1}}
\newcommand{\reft}[1]{~\ref{t:#1}}
\newcommand{\refc}[1]{~\ref{c:#1}}
\newcommand{\refs}[1]{~\ref{s:#1}}
\newcommand{\refb}[1]{~\cite{b:#1}}
\newcommand{\refe}[1]{~\eqref{e:#1}}

\title{Refined bound for sum-free sets \break in groups of prime order}

\author{Jean-Marc Deshouillers}
\email{jean-marc.deshouillers@math.u-bordeaux1.fr}
\address{Bordeaux University \& CNRS, IMB UMR 5251, 33405 Talence, France.}

\author{Vsevolod F. Lev}
\email{seva@math.haifa.ac.il}
\address{Department of Mathematics, University of Haifa at Oranim,
  Tivon 36006, Israel.}


\begin{document}
\baselineskip=16pt

\begin{abstract}
Improving upon earlier results of Freiman and the present authors, we show
that if $p$ is a sufficiently large prime and $A$ is a sum-free subset of the
group of order $p$, such that $n:=|A|>0.318p$, then $A$ is contained in a
dilation of the interval $[n,p-n]\pmod p$.
\end{abstract}

\maketitle

\section{Introduction}

The subset $A$ of an additively written semigroup is called \emph{sum-free}
if there do not exist $a_1,a_2,a_3\in A$ with $a_1+a_2=a_3$; equivalently, if
$A$ is disjoint with its \emph{sumset} $A+A:=\{a_1+a_2\colon a_1,a_2\in A\}$.
Introduced by Schur in 1916 (``the set of positive integers cannot be
partitioned into finitely many sum-free subsets''), sum-free sets become now
a classical object of study in additive combinatorics; we refer the reader to
\cite{b:df,b:l2} and the papers, cited there, for the history and overview of
the subject area.

Let $G$ be a finite abelian group. It is easy to see that a randomly chosen
``small'' subset of $G$ is sum-free with high probability, while a randomly
chosen ``large'' subset of $G$ with high probability is \emph{not} sum-free.
Thus, small sum-free subsets of $G$ can be unstructured, whereas large
sum-free subsets possess a rigid structure. Unraveling this structure for
various underlying groups $G$ is a fascinating problem which received much
attention during the last decade.

In the present paper we consider groups of prime order $p$, which we identify
with the quotient group $\Zp$. Let $\phi_p$ denote the canonical homomorphism
from $\Z$ onto $\Zp$, and for a set $\cS\seq\Z$ let $\cS_p$ denote the image
of $\cS$ under $\phi_p$; here the letter $\cS$ will often be substituted by
the interval notation so that, for instance, $[3,6)_{11}=\{-8,4,16\}_{11}$
etc. The well-known Cauchy-Davenport inequality implies readily that if
$A\seq\Zp$ is sum-free, then $|A|\le\lfl (p+1)/3\rfl$. This estimate is
sharp, as for $u=\lfl(p+1)/3\rfl$ the set $[u,2u-1]_p$, and consequently its
dilates, are sum-free.

The main results of both \refb{l2} and \refb{df} show that in fact for prime
$p$, any large sum-free subset of $\Zp$ is close to a dilate of
$(p/3,2p/3)_p$. Specifically, it is proved in \refb{l2} for $\alp_0=0.33$,
and in \refb{df} for $\alp_0=0.324$ and $p$ large enough that if $A$ is a
sum-free subset of $\Zp$ with $n:=|A|>\alp_0 p$, then $A$ is contained in a
dilate of $[n,p-n]_p$. (As shown in \refb{l2}, the interval $[n,p-n]_p$ is
best possible in this context.)

For an integer $d$ and a subset $A$ of an abelian group let
 $d\ast A:=\{d a\colon a\in A\}$. The goal of the present paper is to prove
\begin{theorem}\label{t:main}
Let $p$ be sufficiently large a prime and suppose that $A\seq\Zp$ is
sum-free. If $n:=|A|>0.318 p$, then there exists $d\in\Z$ such that
 $A\seq d\ast[n,p-n]_p$.
\end{theorem}

The seemingly modest improvement of the constant from $0.324$ to $0.318$
requires a substantial effort and a number of new ideas,  some at the level
of Fourier analysis and others of a combinatorial nature; we believe that
these ideas may actually be of more general interest than the improvement of
the constant itself.

An example, presented in \refb{l3}, shows that the constant in question
cannot be reduced to below $0.2$. Though the value $0.318$ is not the precise
limit of our method, narrowing significantly the gap between $0.2$ and
$0.318$ seems to be a rather non-trivial and exciting problem.

\section{Some lemmas}\label{s:lemmas}

We gather here several auxiliary results, used in the next section to prove
Theorem \reft{main}.

It is well-known that if a set is sum-free, then its characteristic function
has a large Fourier coefficient. Specifically, let $\e_p$ denote the
character of the group $\Zp$, defined by $\e_p(1):=\exp(2\pi i/p)$, and given
a set $A\seq\Zp$ and an integer $z$ write $\hA(z):=\sum_{a\in A} \e_p(az)$. A
standard argument shows that if $A\seq\Zp$ is sum-free with $|A|>\alp_0p$,
then there exists $z\in\Z$ with $\phi_p(z)\neq 0$ such that
$|\hA(z)|>\frac{\alp_0^2}{1-\alp_0}\,p$. For $\alp_0=0.318$ this leads to
$|\hA(z)|>0.148p$, while the following lemma allows us to get
$|\hA(z)|>0.152p$.
\begin{lemma}\label{l:intrick}
Let $\kap>0$ and $\gam\in(0,1)$ be real numbers, and write
 $K:=\lfl \gam^{-1/\kap}\rfl$. Suppose that $P>K$ is an integer and
$v_1\longc v_P$ are non-negative real numbers, satisfying
  $$ \textstyle
     \sum_{i=1}^P v_i = 1 \quad\text{and}\quad
                          \sum_{i=1}^P v_i^{1+\kap} \ge \gam. $$
Then the equation $Kx^{1+\kap}+(1-Kx)^{1+\kap}=\gam$ in the variable $x$ has
exactly one solution in the interval $(1/(K+1),1/K]$, and denoting this
solution by $X$ we have
  $$ \max \{ v_1\longc v_P \} \ge X. $$
\end{lemma}

\begin{proof}
The existence and uniqueness of a solution is an immediate consequence of the
intermediate value property: just notice that the function
$f(x):=Kx^{1+\kap}+(1-Kx)^{1+\kap}$ is continuous and increasing on
$[1/(K+1),1/K]$, and that
  $$ f(1/(K+1)) = \frac1{(K+1)^\kap} < \gam \le \frac1{K^\kap} = f(1/K). $$
We now prove the second assertion. Let $W$ denote the set of all those real
vectors $(w_1\longc w_P)\in\R^P$ with non-negative coordinates, satisfying
  $$ \textstyle \sum_{i=1}^P w_i = 1 \quad\text{and}\quad
                                      \sum_{i=1}^P w_i^{1+\kap} \ge \gam $$
(so that $(v_1\longc v_P)\in W$). Observing that $W$ is compact and
$\max\{w_1\longc w_P\}$ is a continuous function on $W$, set
  $$ \mu := \min \{ \max\{w_1\longc w_P\}
                                      \colon (w_1\longc w_P) \in W \} $$
and
  $$ \cC := \{ (w_1\longc w_P) \in W
                                  \colon \max\{w_1\longc w_P\} = \mu \}. $$

We notice that if $(w_1\longc w_P)\in\cC$, then
  $$ \gam \le \sum_{i=1}^P w_i^{1+\kap}
                           \le \mu^{\kap} \sum_{i=1}^P w_i = \mu^{\kap}, $$
whence
\begin{equation}\label{e:floormu}
  \frac1\mu \le \gam^{-1/\kap} < K+1.
\end{equation}
On the other hand, it is readily verified that if $w_i=1/K$ for $i\in[1,K]$
and $w_i=0$ for $i\in[K+1,P]$, then $(w_1\longc w_P)\in W$, and hence
 $\mu \le 1/K$, implying $1/\mu\ge K$. Comparing this with \refe{floormu} we
derive that

\begin{equation}\label{e:floor}
  \lfl 1/\mu \rfl = K.
\end{equation}

For $\del\in\R$ and $i,j\in[1,P]$ with $i\neq j$, define the operator
$T_{ij}^{(\del)}\colon\R^P\to\R^P$ by
 $T_{ij}^{(\del)}(w_1\longc w_P)=(w_1'\longc w_P')$, where
$w_i'=w_i+\del,\ w_j':=w_j-\del$, and $w_k'=w_k$ for $k\in[1,P]\stm\{i,j\}$.
Observe that if $w_i<w_j$ and $0<\del<(w_j-w_i)/2$, then $\sum_{i=1}^P
(w_i')^{1+\kap}<\sum_{i=1}^P w_i^{1+\kap}$.

Note, that if $(w_1\longc w_P)\in\cC$, then not all coordinates $w_i$ are
equal to each other: else they all would be equal to $1/P$, implying
 $\gam\le P\cdot(1/P)^{1+\kap}=P^{-\kap}$ and hence contradicting
$P\ge K+1>\gam^{-1/\kap}$.

We claim now that if $(w_1\longc w_P)\in\cC$, then equality holds in
$\sum_{i=1}^P w_i^{1+\kap}\ge\gam$. Indeed, assuming that
 $\sum_{i=1}^P w_i^{1+\kap}>\gam$, find $i,j\in[1,P]$ with $w_i<w_j=\mu$ and
apply to $(w_1\longc w_P)$ the transformation $T_{ij}^{(\del)}$ with
$\del\in(0,(w_j-w_i)/2)$ small enough to ensure that the resulting vector
$(w_1'\longc w_P')$ satisfies $\sum_{i=1}^P (w_i')^{1+\kap}>\gam$. Repeating
this procedure sufficiently many times, we find eventually a vector
$(u_1\longc u_P)\in W$ with $\max\{u_1\longc u_P\}<\mu$, contradicting the
definition of $\mu$.

Next, we observe that for any $(w_1\longc w_P)\in\cC$ there is at most one
index $i\in[1,P]$ such that $0<w_i<\mu$. For if $0<w_i\le w_j<\mu$, where
$i,j\in[1,P]$ are distinct, then, applying to $(w_1\longc w_P)$ the
transformation $T_{ij}^{(\del)}$ with $\del$ negative and sufficiently small
in absolute value, we obtain a vector $(w_1'\longc w_P')\in\cC$ with
$\sum_{i=1}^P (w_i')^{1+\kap}>\gam$; however, we showed above that this is
impossible.

Fix $(w_1\longc w_P)\in\cC$. As it follows from our last observation, there
is an integer $k\in[1,P-1]$ such that, re-ordering the coordinates of
$(w_1\longc w_P)$, if necessary, we can write
  $$ \mu = w_1\longe w_k > w_{k+1} \ge 0 = w_{k+2}\longe w_P. $$
From
  $$ k\mu \le \sum_{i=1}^P w_i < (k+1)\mu $$
it follows then that $k\le 1/\mu<k+1$, whence $k=\lfl 1/\mu\rfl=K$ by
\refe{floor}, and consequently $w_{k+1}=1-K\mu$. This yields
  $$ K\mu^{1+\kap}+(1-K\mu)^{1+\kap}=\gam, $$
so that in fact $\mu=X$, implying the second assertion of the lemma and
indeed, showing that the estimate of the lemma is sharp.
\end{proof}

As indicated at the beginning of this section, Lemma \refl{intrick} will be
used to show that if $A\seq\Zp$ is sum-free with $|A|>0.318p$, then there
exists $z\in\Z$ with $\phi_p(z)\neq 0$ and such that $|\hA(z)|>0.152p$. A
well-known result of Freiman leads then to the conclusion that there is an
interval of the form $[u,u+p/2)_p$, with an integer $u$, containing at least
$(|A|+|\hA(z)|)/2>0.235p$ elements of the dilation $z\ast A$. Our next lemma,
which is a reformulation of \cite[Corollary 2]{b:l1}, allows us to improve
this to $0.238p$.
\begin{lemma}[\protect{\cite[Corollary 2]{b:l1}}]\label{l:frep}
Let $p$ be a positive integer and suppose that $A\seq\Zp$. If $n=|A|$ and
$S=\sum_{a\in A} \e_p(a)$, then there exists an integer $u$ such that
  $$ |A\cap[u,u+p/2)_p| \ge \frac{n}2\,
             + \frac{p}{2\pi}\,\arcsin\Big( |S|\sin \frac\pi{p} \Big). $$
\end{lemma}

For a subset $A$ of an additively written abelian group write
  $$ A-A := \{ a_1-a_2 \colon a_1,a_2\in A \}. $$

The following lemma follows readily from the results of \refb{gaf}; see also
\cite[Theorem~2]{b:ls}.

\begin{lemma}\label{l:3n-3}
Let $\ell$ and $m$ be positive integers and suppose that $A\seq[0,\ell]$ is a
set of integers such that $|A|=m,\ 0\in A,\ \ell\in A$, and $\gcd(A)=1$. Then
  $$ |A-A| \ge \min \{ \ell + m, 3m - 3 \}. $$
\end{lemma}

The next two lemmas deal with the structure of the difference set $A-A$ in
the case where $A$ is a dense set of integers.

\begin{lemma}[\protect{\cite[Lemma 2]{b:l2}}]\label{l:dif2}
Let $m$ and $\ell$ be positive integers, satisfying $\ell\le 2m-2$, and
suppose that $A\seq[0,\ell]$ is a set of integers such that $|A|=m$. Then for
any integer $k\ge 1$ we have
  $$ \left( \frac{\ell-m+1}k,\,\frac{m}k \right) \seq A-A. $$
\end{lemma}

\begin{lemma}[\protect{\cite[Lemma 3]{b:l2}}]\label{l:dif3}
Let $m$ and $\ell$ be positive integers and suppose that $A\seq[0,\ell]$ is a
set of integers such that $|A|=m$. If $\ell<\frac{2k-1}k\,m-1$ with an
integer
 $k\ge 2$, then
   $$ \left(-\frac m{k-1},\frac m{k-1}\right) \seq A-A. $$
\end{lemma}

We notice that Lemmas \refl{dif2} and \refl{dif3} remain valid if $A$ is a
subset of $\Zp$ (instead of $\Z$), the condition $A\seq[0,\ell]$ is replaced
by $A\seq[u,u+\ell]_p$ with integers $u$ and $\ell<p$, and the intervals in
the conclusions of the lemmas are replaced by their images under $\phi_p$.
Similarly, the estimate of Lemma \refl{3n-3} remains valid if
$A\seq[u,u+\ell]_p$ with integer $u$ and $\ell<p/2$, and given that the set
$\phi_p^{-1}(A)\cap[u,u+\ell]$ is not contained in an arithmetic progression
of length, smaller than $\ell$.

The next lemma is a restatement of a particular case of a $\Z/p\Z$-version of
\cite[Lemma 3]{b:df}.
\begin{lemma}\label{l:EZ}
Let $p$ be a prime and let $1\le\ell<p$ and $u$ be integers. Suppose that
$A\seq\Zp$ is sum-free and that $A_0\seq[u,u+\ell]_p\cap A$, and write
$m:=|A_0|$. If $\ell\le 2m-2$, then for any integer $a\in[\ell/4,\ell/2]$
with $\phi_p(a)\in A$ we have
  $$ [2a-(2m-\ell-2), 2a+(2m-\ell-2)]_p \cap (A \cup (-A)) = \est. $$
\end{lemma}

For the convenience of the reader we provide a proof.

\begin{proof}[Proof of Lemma \refl{EZ}]
Since $|\{z,z+a\}_p\cap A|\le 1$ for any integer $z$, the set $A_0$ has at
most $a$ elements in each of the intervals $[u,u+2a-1]_p$ and
$[u+(\ell-2a+1),u+\ell]_p$, and consequently we have
\begin{equation}\label{e:parti}
  |A_0\cap [u+2a,u+\ell]_p|\ge m-a\quad \text{and}
                                    \quad |A_0\cap [u,u+(\ell-2a)]_p|\ge m-a.
\end{equation}
Assuming now that there exists an integer $x\in
[2a-(2m-\ell-2),2a+(2m-\ell-2)]$ with $\phi_p(x)\in A\cup(-A)$, we will
obtain a contradiction.

Suppose first that $x>2a$ and consider in this case the two-element sets
  $$ \{u,u+x\}_p,\ \{u+1,u+1+x\}_p,\ \ldots,\ \{u+\ell-x,u+\ell\}_p. $$
(We notice that \refe{parti}, along with $a\le \ell/2<m$, implies that
 $m-a\le \ell-2a+1$, whence $\ell\ge a+m-1$ and consequently,
$\ell-x\ge \ell-(2a+(2m-\ell-2))=2(\ell-a-m+1)\ge 0$.) These sets are
pairwise disjoint (as $u+\ell-x<u+x$ in view of $2x>4a\ge \ell$) and they all
are contained in $[u,u+\ell-2a]_p\cup[u+2a,u+\ell]_p$. Since at most one
element out of each of these $\ell-x+1$ sets belongs to $A_0$, we conclude
that
  $$ |[u,u+\ell-2a]_p\stm A_0| + |[u+2a,u+\ell]_p\stm A_0| \ge \ell-x+1. $$
Therefore
\begin{align*}
  |[u,u+\ell-2a]_p\cap A_0| + |[u+2a,u+\ell]_p\cap A_0|
    & \le 2(\ell-2a+1)-(\ell-x+1) \\
    & \le \ell-4a + (2a+(2m-\ell-2)) +1 \\
    & = 2(m-a) - 1,
\end{align*}
contradicting \refe{parti}.

Similarly, if $x<2a$, then we obtain a contradiction with \refe{parti}
considering the $l-4a+x+1$ sets
\begin{multline*}
  \{u+2a-x,u+2a\}_p,\ \{u+2a-x+1,u+2a+1\}_p, \\
    \ldots,\ \{u+l-2a,u+l-2a+x\}_p
\end{multline*}
which, again, are pairwise disjoint and contained in
$[u,u+l-2a]_p\cup[u+2a,u+l]_p$.
\end{proof}

\begin{lemma}\label{l:p-n3}
Let $p$ be a prime and suppose that $A\seq\Zp$ is sum-free. Write $n:=|A|$.
If
  $$ [-(p-n+1)/3,(p-n+1)/3]_p \cap A = \est, $$
then $A\seq[n,p-n]_p$.
\end{lemma}

\begin{proof}
Set
\begin{equation}\label{e:mu}
  \mu := \min \{ |z| \colon z\in\Z,\,\phi_p(z)\in A \},
\end{equation}
so that $\mu>(p-n+1)/3$. Clearly, for any $a\in[\mu,2\mu)_p\cap A$ we have
$a+\mu\in[2\mu,3\mu)_p\stm A$, which gives
  $$ |[\mu,3\mu)_p \cap A| \le \mu. $$
Assuming that $\mu<p/4$ we get then
  $$ n \le \mu + |[3\mu,p-\mu]_p| \le \mu + (p-4\mu+1) = p - 3\mu + 1 $$
whence $3\mu\le p-n+1$, contradicting the assumptions. We have therefore
$\mu>p/4$ and then $3\mu>p-\mu$, implying
  $$ n = |A\cap[\mu,3\mu)_p| \le \mu; $$
that is, $[0,n-1]_p\cap A=\est$. In a similar way (or applying the argument
above to the set $-A:=\{-a\colon a\in A\}$) we obtain
 $[p-(n-1),p]_p\cap A=\est$. The result follows.
\end{proof}

\section{Proof of Theorem \reft{main}}\label{s:mainproof}

Suppose that $p$ is a prime and $A\seq\Zp$ is a sum-free set with
$n:=|A|>0.318p$. The computations below tacitly assume that $p$ is
sufficiently large.

Recalling the definition of $\hA(z)$ from the beginning of Section
\refs{lemmas}, we start with

\begin{claim} There exists an integer $z_0$ with $\phi_p(z_0)\neq 0$ such that
$|\hA(z_0)|>0.152p$.
\end{claim}

\begin{proof}
Let $\alp:=n/p$, so that $\alp>0.318$. By the Parseval identity we have
\begin{equation}\label{e:parseval}
  \sum_{z=1}^{p-1} |\hA(z)|^2 = \alp(1-\alp)p^2.
\end{equation}
Using the fact that $|\hA(p-z)|=|\hA(z)|$ for any $z\in\Z$ and letting
$P:=(p-1)/2$ we re-write \refe{parseval} as
  $$ \sum_{z=1}^P \frac{|\hA(z)|^2}{\alp(1-\alp)p^2/2} = 1. $$
On the other hand, since $A$ is sum-free we have
  $$ \sum_{z=0}^{p-1} |\hA(z)|^2 \hA(z) = 0, $$
whence
  $$ \sum_{z=1}^{p-1} |\hA(z)|^3
         \ge - \sum_{z=1}^{p-1} |\hA(z)|^2\hA(z)
                                             = |\hA(0)|^3 = \alp^3 p^3 $$
and consequently
  $$ \sum_{z=1}^P \lpr \frac{|\hA(z)|^2}{\alp(1-\alp)p^2/2} \rpr^{3/2}
       \ge \sqrt{2} \lpr \frac{\alp}{1-\alp} \rpr^{3/2} > 0.4502. $$
Applying Lemma \refl{intrick} with $\kap=0.5$ and $\gam=0.4502$ (which leads
to $K=4$), we conclude that there exists $z_0\in[1,P]$ with
  $$ \frac{|\hA(z_0)|^2}{\alp(1-\alp)p^2/2} > 0.2131 $$
and accordingly
  $$ |\hA(z_0)| > \sqrt{0.2131\,\alp(1-\alp)/2} \; p > 0.152p. $$
\end{proof}

Dilating $A$, if necessary, we assume that, in fact,
\begin{equation}\label{e:Ahat1}
   |\hA(1)| > 0.152p.
\end{equation}

Choose an integer $u_0$ such that
the number of elements of $A$ in
$[u_0,u_0+p/2)_p$ is maximized, set $A_0:=A\cap[u_0,u_0+p/2)_p$ and
$m:=|A_0|$, and let $B_0:=\phi_p^{-1}(A_0)\cap[u_0,u_0+p/2)$. Furthermore,
put $\ell:=\max B_0 - \min B_0$; thus $A_0$ is contained in a block of
$\ell+1$ consecutive elements of $\Zp$ and
  $$ m = \max \{ |A\cap[u,u+p/2)_p| \colon u\in\Z \}. $$
We notice that the last equality implies that
\begin{equation}\label{e:inhalf}
  n-m \le |A\cap[u,u+p/2)_p| \le m
\end{equation}
for all real $u$.

By Lemma \refl{frep} we have
\begin{equation}\label{e:n0large}
  m \ge \frac n2\, + \frac{p}{2\pi}
      \,\arcsin\Big( |\hA(1)| \sin \frac\pi{p} \Big)
          > \Big( 0.159 + \frac1{2\pi} \arcsin(0.1519\pi) \Big) p > 0.238 p.
\end{equation}
Since $B_0$ is a subset of an interval of length $\ell<p/2$, this shows that
$B_0$ is not contained in an arithmetic progression with difference greater
than $2$, and we now dispose of the case where $B_0$ is contained in an
arithmetic progression with difference $2$.

\begin{claim}\label{c:gcd2}
If $B_0$ is contained in an arithmetic progression with difference $2$, then
the conclusion of the theorem holds true.
\end{claim}

\begin{proof}
If $B_0$ is contained in an arithmetic progression with difference $2$, then
there is an integer $u$ and a set $C\seq\Zp$ such that $C\seq[u,u+p/4)_p$ and
either $A_0=2\ast C$, or $A_0=2\ast C+1$. Evidently, we have $\lfl
p/4\rfl<\frac32\,m-1$, whence $(-m,m)_p\seq C-C$ by Lemma \refl{dif3} (see
also the remark after the lemma), implying $2\ast(-m,m)_p\seq A_0-A_0$. Since
$A$ is sum-free, we derive that the set $2\ast(-m,m)_p$ is disjoint with $A$,
and replacing $A$ with its dilation by the factor $(p-1)/2$ we obtain
$A\seq[m,p-m]_p$. The assertion now follows from Lemma \refl{p-n3}, as
  $$ (p-n+1)/3 < 0.228 p < m $$
by \refe{n0large}.
\end{proof}

In what follows we assume that $B_0$ is \emph{not} contained in an arithmetic
progression with difference greater than $1$.

Since the sets $A$ and $A-A$ are disjoint, we have
  $$ n \le p - |A-A| \le p - |A_0-A_0|. $$
To estimate $|A_0-A_0|$ we apply Lemma \refl{3n-3}; this gives
\begin{equation}\label{e:diffr}
  p-n \ge |A_0-A_0| \ge \min \{ \ell+m, 3m-3 \}.
\end{equation}
Assuming that $\ell\ge 2m-3$ and using \refe{n0large} we then obtain
  $$ p \ge n + 3m - 3 > (0.318+3\cdot0.238)p - 3 = 1.032p - 3, $$
a contradiction. Thus
\begin{gather}
  \ell \le 2m-4, \label{e:l0small} \\
  |A_0-A_0| \ge \ell+m, \label{e:A0-A0}
\end{gather}
and $p-n \ge \ell+m$ by \refe{diffr}, whence
\begin{equation}\label{e:l0smallprime}
  \ell \le p-n-m.
\end{equation}
We assume, furthermore, that
\begin{equation}\label{e:l032n0}
  \ell \ge \frac32\,m-1;
\end{equation}
for otherwise $(-m,m)_p\cap A=\est$ by Lemma \refl{dif3}, and consequently
$A\seq[n,p-n]_p$ (as at the end of the proof of Claim \refc{gcd2}).
Assumption \refe{l032n0} will eventually lead us to a contradiction.

\begin{claim}\label{c:n02excluded}
We have
  $$ A\cap(-m/2,m/2)_p=\est. $$
\end{claim}

\begin{proof}
Let $\mu$ be defined by \refe{mu}; we want to show that $\mu\ge m/2$. Notice,
that by \refe{l0small} and Lemma \refl{dif2} we have
$((\ell-m+1)/2,m/2)_p\seq A_0-A_0\seq A-A$, and consequently
$((\ell-m+1)/2,m/2)_p\cap A=\est$; thus, it actually suffices to prove that
$\mu>(\ell-m)/2$. Assume that this is wrong, and hence
\begin{equation}\label{e:loc20}
  \cos \pi \, \frac\mu p \ge \cos \pi \, \frac{\ell-m}{2p}
        \ge \cos \pi \, \frac{p-n-2m}{2p} = \sin \pi \, \frac{n+2m}{2p}
\end{equation}
holds by \refe{l0smallprime}.

Since $A\cap(A+\mu)=\est$, we have $|A\cup(A+\mu)|=2n$  and
\begin{equation}\label{e:Strick}
   \bigg| \sum_{a\in A\cup(\mu+A)} \e_p(a) \bigg|
                       = \big|(1+\e_p(\mu))\,\hA(1)\big|
                                           = 2 |\hA(1)|\, \cos\pi\frac\mu p.
\end{equation}
We distinguish two cases. Assume first that $m>0.25p$. From \refe{Strick} we
get
  $$ 2 |\hA(1)|\, \cos\pi\frac{\mu}{p}
               \le \bigg| \sum_{z=0}^{2n-1} \e_p(z) \bigg|
                                    = \frac{\sin 2\pi n/p}{\sin \pi/p} $$
which, along with \refe{loc20}, \refe{Ahat1}, and the assumption $m>0.25p$,
implies
\begin{align*}
  \sin2\pi\frac np
    &\ge 2|\hA(1)|\, \cos\pi\frac{\mu}{p} \,\sin\frac\pi{p} \\
    &>   0.954 \, \frac{|\hA(1)|}{0.152p} \, \cdot \frac{\sin\pi/p}{3.141/p}
                                               \cdot \cos\pi\frac{\mu}{p} \\
    &>   0.954 \sin \pi \, \frac{n+2m}{2p} \\
    &>   0.954 \sin \frac{\pi}2\lpr \frac np+0.5\rpr.
\end{align*}
It is easy to verify, however, that the function
  $$ \sin 2\pi x - 0.954 \sin \frac{\pi}2\,(x+0.5) $$
is negative for any $x\in(0.318,0.334)$, a contradiction.

Assume now that $m<0.25p$. In this case we apply Lemma \refl{frep} to the set
$A\cup(A+\mu)$, observing that by \refe{inhalf} any interval of the form
$[u,u+p/2)_p$ with integer $u$ contains at most $2m$ elements of this set; in
view of \refe{Strick} this yields
  $$ 2m \ge n + \frac{p}{2\pi}\,
            \arcsin\lpr 2|\hA(1)|\cos\pi\frac\mu p \, \sin\frac\pi p \rpr  $$
and hence
  $$ 2\pi\frac{2m-n}p
       \ge \arcsin \lpr 2|\hA(1)|\cos\pi\frac\mu p \, \sin\frac\pi p \rpr. $$
Since $2\pi(2m-n)/p\le 2\pi m/p<\pi/2$ we obtain
  $$ \sin 2\pi\frac{2m-n}p
                          \ge 2|\hA(1)|\cos\pi\frac\mu p \, \sin\frac\pi p $$
which, as above, yields
\begin{equation}\label{e:muest}
  \sin 2\pi\frac{2m-n}p > 0.954 \cos\pi\frac\mu p
                                      > 0.954 \sin \pi \, \frac{n+2m}{2p}.
\end{equation}
On the other hand, it is not difficult to see that the function
  $$ \sin 2\pi (2y-x) - 0.954 \sin\frac\pi2\,(x+2y) $$
is negative in the region $x\in(0.318,0.334),\,y\in(0.23,0.25)$, a
contradiction again.
\end{proof}

Set
\begin{alignat*}{5}
  A_1^+ &:= A\cap[m/2,\ell-m+1]_p\,,
       &\quad A_2^+ &:= A\cap[m,p/4)_p\,, \\
  A_1^- &:= A\cap[-(\ell-m+1),-m/2]_p\,, &\quad A_2^- &:= A\cap(-p/4,-m]_p\,,
\intertext{and}
  A_1 &:= A_1^+\cup A_1^-\,, &\quad A_2 &:= A_2^+\cup A_2^-,
\end{alignat*}
so that
\begin{equation}\label{e:A1A2}
  A\cap[0,p/4)_p=A_1^+\cup A_2^+\quad \text{and}
                                      \quad A\cap(-p/4,0]_p=A_1^-\cup A_2^-
\end{equation}
by Lemma \refl{dif2} (applied with $k=1$) and Claim \refc{n02excluded};
observe also that $m/2\le \ell-m+1$ by \refe{l032n0}, and that $A_2=\est$ if
$m>0.25p$. For definiteness, we assume for the rest of the proof that
\begin{equation}\label{e:+>-}
  |A_1^+| \ge |A_1^-|
\end{equation}
and hence $A_1^+\neq\est$: otherwise by \refe{inhalf} we would have
  $$ n-m \le |A_2| \le \max \{ 0, p/2-2m \}, $$
leading to either $m=n$ (in which case we are done by Lemma \refl{p-n3}), or
$n+m<p/2$ (which contradicts \refe{n0large}).

Given two subsets $S_1$ and $S_2$ of an additively written semigroup, we
write
  $$ S_1+S_2 := \{ s_1+s_2\colon s_1\in S_1,\,s_2\in S_2 \}. $$
It is well-known and easy to prove that if $S_1$ and $S_2$ are finite
non-empty sets of integers, then $|S_1+S_2|\ge|S_1|+|S_2|-1$ holds. Clearly,
this inequality remains valid also if $S_1$ and $S_2$ are non-empty subsets
of $\Zp$, contained in two intervals of total length, smaller than $p$.

Our next claim refines the estimate \refe{l0smallprime}.
\begin{claim}\label{c:l0verysmall}
We have
  $$ \ell \le p - n - m - 2|A_2| + 2. $$
\end{claim}

\begin{proof}
The assertion follows from the fact that the sets $A$ and
  $$ A_2^++A_2^+ \seq [2m,p/2)_p\,,
                \quad A_2^-+A_2^- \seq (p/2,p-2m]_p\,,
                                            \quad A_0-A_0\seq[-\ell,\ell] $$
are pairwise disjoint, the estimate \refe{A0-A0}, and the observation that
$|A_2^++A_2^+|\ge 2|A_2^+|-1$ and $|A_2^-+A_2^-|\ge 2|A_2^-|-1$.
\end{proof}

Write $I:=[-(2m-\ell-2),2m-\ell-2]_p$ and
 $J:=2\ast A_1^++I$. We observe that
\begin{equation}\label{e:caseI2}
  J \cap A = (-J)\cap A = \est
\end{equation}
by Lemma \refl{EZ} (which is applicable since
$A_1^+\seq[m/2,\ell-m+1]_p\seq[\ell/4,\ell/2]_p$, as it follows from
\refe{l0small}). Let $k:=|A_1^+|$, write $A_1^+=\{a_1\longc a_k\}$, where the
elements are so numbered that their inverse images in $[m/2,\ell-m+1]$ under
$\phi_p$ form an increasing sequence, and for $i\in[1,k]$ set
$S_i:=2\ast\{a_1\longc a_i\}+I$. We have then $S_1=|I|=4m-2\ell-3$ and
$|S_{i+1}\stm S_i|\ge 2$ for $i\in[1,k-1]$, and it follows that
\begin{equation}\label{e:Slarge}
   |J| = |S_k| \ge (4m-2\ell-3) + 2(k-1) = 2|A_1^+|+4m-2\ell-5.
\end{equation}

We are now in a position to complete the proof showing that the above-made
assumptions (see the remark following \refe{l032n0}) lead to a contradiction.
We consider separately two cases: $m<0.244p$ and $m>0.244p$.

\subsection*{Case I: $m<0.244p$}
We revisit the proof of Claim \refc{n02excluded}, defining $\mu$ by \refe{mu}
and observing that \refe{muest} gives
\begin{equation}\label{e:arccos}
  \mu > \frac1\pi\, \arccos
                \Big( 1.049 \sin 2\pi\Big( \frac{2m}p-0.318 \Big) \Big) p.
\end{equation}
(This estimate is stronger, than $\mu\ge m/2$, for small values of $m$, and
in particular for $m<0.244p$.)

Since $A_1^+\seq[\mu,\ell-m+1]_p$, we have
  $$ J \seq [\ell-2m+2\mu+2,\ell]_p, $$
hence \refe{+>-}, \refe{caseI2} and \refe{Slarge} yield
\begin{align*}
  |[\ell-2m+2\mu+2,\ell]_p\stm A| &\ge 2|A_1^+|+4m-2\ell-5 \\
\intertext{and}
  |[-\ell,-(\ell-2m+2\mu+2)]_p\stm A| &\ge 2|A_1^+|+4m-2\ell-5 \\
                                      &\ge 2|A_1^-|+4m-2\ell-5.
\end{align*}

Adding up these estimates we obtain
  $$ |[\ell-2m+2\mu+2, p-(\ell-2m+2\mu+2)]_p \stm A|
                                               \ge 2|A_1|+8m-4\ell-10, $$
and it follows that
\begin{equation}\label{e:caseI5}
  |[\ell-2m+2\mu+2, p-(\ell-2m+2\mu+2)]_p \cap A|
                                           \le p-2|A_1|+2\ell-4m-4\mu+7.
\end{equation}
We notice now that
\begin{multline*}
  \ell - 2m + 2\mu + 1 \le \ell - m + 1 \le p - n - 2m + 1 \\
                                  < (1-0.318-2\cdot 0.238)p + 1 < 0.25p
\end{multline*}
by \refe{l0smallprime} and \refe{n0large}, and consequently
  $$ |[-(\ell-2m+2\mu+1),\ell-2m+2\mu+1]_p\cap A| \le |A_1|+|A_2| $$
holds. Taking the sum of the last inequality and \refe{caseI5}, observing
that $|A_1|+|A_2|\ge n-m$ by \refe{inhalf}, and using the estimate of Claim
\refc{l0verysmall}, we get
\begin{align*}
  n &\le p - |A_1| + |A_2| + 2\ell - 4m - 4\mu + 7 \\
    &\le 3p - |A_1| - |A_2| - 2n -6m - 4\mu + 11 \\
    & \le 3p - 3n - 5m -4\mu + 11.
\end{align*}
This yields
  $$ \frac{4\mu}{p} \le 3 - \frac{4n}p - \frac{5m}p + \frac{11}p
                                               < 1.729-\frac{5m}p, $$
and comparing this with \refe{arccos} we obtain
  $$ \frac4\pi\, \arccos \Big( 1.049
             \sin 2\pi\Big( \frac{2m}p-0.318 \Big) \Big)
                                               < 1.729-\frac{5m}p. $$
However, a routine investigation shows that
  $$ \frac4\pi\, \arccos\big( 1.049\sin 2\pi (2x-0.318) \big) - 1.729 + 5x $$
is positive for $x\in(0.238,0.244)$.

\subsection*{Case II: $m>0.244p$}
Recalling \refe{+>-} we get
\begin{equation}\label{e:A1large}
  |A_1^+| \ge \frac12\,{|A_1|}
               = \frac12\, \big( |(-p/4,p/4)_p\cap A| - |A_2| \big)
                                            \ge \frac12\, ( n-m - |A_2|)
\end{equation}
by \refe{inhalf} and the definitions of $A_1$ and $A_2$. We also observe that
\begin{align}
  |A_1^+| &\le (\ell-m+1) - m/2 + 1 \nonumber \\
          &\le p-n-\frac52\,m + 2 \nonumber \\
          &=   (n-m) - \Big( 2n+\frac32\,m-p \Big) + 2 \nonumber \\
          &<   (n-m) - (2\cdot0.318+1.5\cdot0.244-1)p + 2 \nonumber \\
          &<   n-m \label{e:A1small}
\end{align}
by \refe{l0smallprime}.

Using Claim \refc{l0verysmall} we derive from \refe{A1large} that
\begin{align*}
  |[m/2,\ell-m+1]_p\stm A_1^+|
     &\le \Big( \ell-\frac32\,m+2 \Big) - \frac12\,(n - m - |A_2|) \\
     &=   \ell - m - \frac12\,n + \frac12\,|A_2| + 2 \\
     &=   2(2m-\ell-2) + 3\ell - 5m - \frac12\,n + \frac12\,|A_2| + 6 \\
     &\le 2(2m-\ell-2) + 3p - 8m - \frac72\,n + 12 \\
     &<   2(2m-\ell-2) - (8\cdot 0.238 + 3.5\cdot 0.318 - 3)p + 12 \\
     &<   2(2m-\ell-2),
\end{align*}
and it follows that the set $J$ is an interval in $\Zp$. Consequently, by
\refe{inhalf}, Claim \refc{n02excluded}, the definition of $A_1^+$,
\refe{A1A2}, and \refe{caseI2} we have
\begin{align}
  n-m &\le |(-m/2,(p-m)/2)_p\cap A| \nonumber \\
        &=   |A_1^+| + |[m,(p-m)/2)_p\cap A| \nonumber \\
        &\le |A_1^+| + |[m,(p-m)/2)_p\stm J|. \label{e:stmJ}
\end{align}
Since $J\seq[\ell-m+2,\ell]_p$ (as it is immediate from the definitions of
$J$ and $A_1^+$), we have
  $$ |[m,(p-m)/2)_p\stm J| \le \max \{ |[m,(p-m)/2)_p\stm J_1|,
                                       \: |[m,(p-m)/2)_p\stm J_2|\}, $$
where
  $$ J_1 := [\ell-m+2,\ell-m+|J|+1]_p
                          \quad \text{and}\quad J_2 := [\ell+1-|J|,\ell]_p $$
(so that $J_1$ and $J_2$ are subintervals of the interval $[\ell-m+2,\ell]_p$
with $|J_1|=|J_2|=|J|$, adjacent to the endpoints of this interval).
Accordingly, from \refe{stmJ} we deduce that
\begin{equation}\label{e:J12}
  |[m,(p-m)/2)_p\stm J_i| \ge n-m-|A_1^+|
\end{equation}
holds true with either $i=1$, or $i=2$.

By \refe{A1small}, assuming that \refe{J12} holds with $i=1$ we obtain
  $$ n-m-|A_1^+| < (p-m)/2 - (\ell-m+|J|+1) + 1.  $$
Using \refe{Slarge}, \refe{A1large}, and Claim \refc{l0verysmall} we now get
\begin{gather*}
  |A_1^+| + 4m - 2\ell - 5 \le |J|-|A_1^+|
                                    < 0.5p + \frac32\,m - n - \ell, \\
  |A_1^+| < 0.5p - \frac52\,m - n + \ell + 5
                                 \le 1.5p - \frac72\,m - 2n -2|A_2| + 7, \\
  \frac52\,n + 3m < 1.5p + 7,
\end{gather*}
which is wrong since
  $$ \frac52\,n + 3m > (2.5\cdot0.318+3\cdot0.238)p = 1.509p. $$

Assume now that \refe{J12} holds with $i=2$. Since
  $$ \frac12\,(n-m-|A_2|) \le |A_1^+| \le \ell-\frac32\,m + 2 $$
by \refe{A1large} and the definition of $A_1^+$, we have
  $$ \ell \ge \frac12\,n + m - \frac12\,|A_2| - 2
       = \frac12\,(p-m) + \frac12\,(n + 3m - p - |A_2|) - 2
                                                     > \frac12\,(p-m), $$
as $|A_2|\le\max\{0,0.5p-2m+2\}$ and
\begin{gather*}
  n + 3m > (0.318+3\cdot 0.238)p = 1.032 p, \\
  n + 5m > (0.318+5\cdot0.238)p = 1.508 p.
\end{gather*}
Similarly to the above we now obtain from \refe{J12} and \refe{A1small}
  $$ n - m - |A_1^+| \le (\ell + 1 - |J|) - m + 1 $$
and using \refe{Slarge}, \refe{A1large}, and Claim \refc{l0verysmall} derive
that
\begin{gather*}
  |A_1^+| + 4m - 2\ell - 5 \le |J|-|A_1^+| \le \ell - n + 2, \\
  |A_1^+| \le 3\ell - 4m - n + 7 \le 3p - 7m - 4n -6|A_2| + 13, \\
  3p + 13 \ge \frac92\,n + \frac{13}2\,m
                             > (4.5\cdot 0.318 + 6.5\cdot0.244)p = 3.017p,
\end{gather*}
a contradiction, as wanted.

\section*{Acknowledgement}

The authors are grateful to Dr. K. Srinivas for helpful discussions.

\bigskip


\begin{thebibliography}{DF06}

\bibitem[DF06]{b:df}
  {\sc J.-M.~Deshouillers} and {\sc G. A.~Freiman},
  On sum-free sets modulo $p$,
  \emph{Functiones et Approximatio} {\bf XXXV} (2006), 7--15.

\bibitem[F62]{b:gaf}
  {\sc G. A. Freiman},
  Inverse problems in additive number theory, VI.
  On the addition of finite sets, III.
  \emph{Izv. Vyss. Ucebn. Zaved. Matematika} {\bf 3} (28) (1962), 151--157 (in Russian).

\bibitem[L06a]{b:l1}
  {\sc V.F.~Lev},
  Distribution of points on arcs,
  \emph{Integers} {\bf 5} (2) (2005), A11, 6 pp. (electronic).

\bibitem[L06b]{b:l2}
  {\sc \bysame},
  Large sum-free sets in $\Z/p\Z$,
  \emph{Israel Journal of Math.} {\bf 154} (2006), 221--234.

\bibitem[L]{b:l3}
  {\sc \bysame},
  On sum-free subsets of the torus group,
  \emph{Functiones et Approximatio}, to appear.

\bibitem[LS95]{b:ls}
  {\sc V.F.~Lev} and {\sc P.Y.~Smeliansky},
  On addition of two distinct sets of integers,
  \emph{Acta Arithmetica} {\bf 70} (1) (1995), 85--91.

\end{thebibliography}
\end{document}